\documentclass{amsart}

\usepackage[utf8]{inputenc}
\usepackage[T1]{fontenc}
\usepackage{verbatim}
\usepackage{graphicx}
\usepackage{graphicx,caption2,psfrag,float,color}
\usepackage{amssymb}
\usepackage{amscd}
\usepackage{amsmath}

\usepackage[T1]{fontenc}

\newtheorem{theorem}{Theorem}[section]

\newtheorem{lemma}[theorem]{Lemma}

\numberwithin{equation}{subsection}
\newtheorem{definition}[theorem]{Definition}

\title{Asymptotics for moments of certain cotangent sums for arbitrary exponents}
\author{Helmut Maier and Michael Th. Rassias}
\date{\today}
\address{Department of Mathematics, University of Ulm, Helmholtzstrasse 18, 89081 Ulm, Germany.}
\email{helmut.maier@uni-ulm.de}
\address{Institute of Mathematics, University of Zurich, CH-8057, Zurich, Switzerland
 \& Institute for Advanced Study, Program in Interdisciplinary Studies,
1 Einstein Dr, Princeton, NJ 08540, USA.}
\email{michail.rassias@math.uzh.ch, michailrassias@math.princeton.edu}
\thanks{}

\begin{document}

 \maketitle
 
\begin{abstract} 
In this paper we extend a result on the asymptotics of moments of certain 
cotangent sums associated to the Estermann and Riemann zeta functions
established in a previous paper for integer exponents to arbitrary positive real exponents.

\end{abstract}
\section{Introduction}
The authors in joint work \cite{mr} and the second author in his thesis \cite{rasthesis}, investigated 
the distribution of cotangent sums
$$c_0\left(\frac{r}{b}\right)=-\sum_{m=1}^{b-1}\frac{m}{b}\cot\left(\frac{\pi m r}{b}\right)$$
as $r$ ranges over the set
$$\{r\::\: (r, b)=1,\ A_0b\leq r\leq A_1 b\}\:,$$
where $A_0$, $A_1$ are fixed with $1/2<A_0<A_1<1$ and $b$ tends to infinity.\\
They could show that 
$$H_k=\int_0^1\left(\frac{g(x)}{\pi}\right)^{2k}dx\:,$$
where
$$g(x)=\sum_{l\geq 1}\frac{1-2\{lx\}}{l}\:,$$
a function that has been investigated by de la Bret\`eche and Tenenbaum \cite{bre, bre2},
as well as Balazard and Martin \cite{balaz1, balaz2}. Bettin \cite{bettin} could replace the interval $(1/2, 1)$ for $A_0$, $A_1$ by the interval $(0,1)$.\\
Improving on a result on the order of magnitude of $\int_0^1g(x)^{2k}dx$
obtained in the paper \cite{mr2}, the authors could obtain the following asymptotics
(Theorem 1.1 of \cite{mrasympt}):

\textit{Let $K\in\mathbb{N}$. There is an absolute constant $C>0$, such that
\[
\int_0^1|g(x)|^Kdx=\frac{e^{\gamma}}{\pi}\:\Gamma(K+1)(1+O(\exp(-CK))),\tag{1.1}
\]
for $K\rightarrow \infty$, where $\gamma$ is the Euler-Mascheroni constant.}\\

The authors are thankful to Goubi Mouloud for the information on the value $e^\gamma/\pi$ of the constant.\\
In this paper we extend the result (1.1) to all positive real values for the exponent $K$.
\begin{theorem}\label{thm11}
Let $K\in\mathbb{R}$, $K>0$. There is an absolute constant $C>0$, such that
\[
\int_0^1|g(x)|^Kdx=\frac{e^{\gamma}}{\pi}\:\Gamma(K+1)(1+O(\exp(-CK))),
\]
for $K\rightarrow \infty$, where $\gamma$ is the Euler-Mascheroni constant.
\end{theorem}
\section{Overview and preliminary results}
Like in previous papers, a crucial role is played by the relation of $g(x)$ to 
Wilton's function, established by Balazard and Martin \cite{balaz2} and results about
operators related to continued fraction expansions due to Marmi, Moussa and Yoccoz \cite{Marmi}.\\
We recall some fundamental definitions and results from \cite{Marmi}. For the
proofs of Lemmas 2.2, 2.4, 2.6 of the present paper, see \cite{mr2}.
\begin{definition}\label{def21}
Let $X=(0,1)\setminus\mathbb{Q}$. Let $\alpha(x)=\{1/x\}$ for $x\in X$. The iterates
$\alpha_k$ of $\alpha$ are defined by $\alpha_0(x)=x$ and 
$$\alpha_k(x)=\alpha(\alpha_{k-1}(x)),\ \text{for}\ k>1.$$
\end{definition}
\begin{lemma}\label{lem22}
Let $x\in X$ and let
$$x=[a_0(x); a_1(x),\ldots, a_k(x),\ldots]$$
be the continued fraction expansion of $x$. We define the partial quotient of 
$p_k(x)$, $q_k(x)$:
$$\frac{p_k(x)}{q_k(x)}:=[a_0(x); a_1(x),\ldots, a_k(x)],\ \text{where},\ (p_k(x), q_k(x))=1\:.$$
Then we have 
$$a_k(x)=\left\lfloor \frac{1}{\alpha_{k-1}(x)}\right\rfloor\:,$$
$$p_{k+1}=a_{k+1}p_k+p_{k-1}$$
and
$$q_{k+1}=a_{k+1}q_k+q_{k-1}\:.$$
\end{lemma}
\begin{definition}\label{def23}
Let $x\in X$. Let also
$$\beta_k(x):=\alpha_0(x)\alpha_1(x)\cdots \alpha_k(x),\ \beta_{-1}(x)=1$$
$$\gamma_k(x):=\beta_{k-1}(x)\log\frac{1}{\alpha_k(x)},\ \text{where}\ k\geq 0,$$
so that $\gamma_0(x):=\log(1/x)$.\\
The number $x$ is called a \textbf{Wilton number} if the series
$$\sum_{k\geq 0}(-1)^k\gamma_k(x)$$
converges.\\
Wilton's function $\mathcal{W}(x)$ is defined by
$$\mathcal{W}(x)=\sum_{k\geq 0}(-1)^k\gamma_k(x)$$
for each Wilton number $x\in (0,1)$.
\end{definition}
\begin{lemma}\label{lem24}
A number $x\in X$ is a Wilton number if and only if $\alpha(x)$ is a 
Wilton number. In this case we have:
$$\mathcal{W}(x)=\log\frac{1}{x}-x\mathcal{W}(\alpha(x)).$$
\end{lemma}
\begin{definition}\label{def25}
Let $p>1$ and $T\::\: L^p\rightarrow L^p$ be defined by 
$$Tf(x):=xf(\alpha(x)).$$
The measure $m$ is defined by 
$$m(\mathcal{E}):=\frac{1}{\log 2}\int_{\mathcal{E}}\frac{dx}{1+x},$$
where $\mathcal{E}$ is any measurable subset of $(0,1)$.
\end{definition}
\begin{lemma}\label{lem26}
Let $p>1$, $n\in \mathbb{N}$.\\
(i) The measure $m$ is invariant with respect to the map $\alpha$, i.e.
$$m(\alpha(\mathcal{E}))=m(\mathcal{E})\:,$$
for all measurable subsets of $\mathcal{E}\subset (0,1)$.\\
(ii) For $f\in L^p$ we have 
$$\int_0^1|T^nf(x)|^p dm(x)\leq g^{(n-1)p}\int_0^1|f(x)|^pdm(x),$$
where 
$$g:=\frac{\sqrt{5}-1}{2}<1.$$
\end{lemma}
\begin{definition}\label{def27}
For $n\in\mathbb{N}$, $x\in X$, we define
$$\mathcal{L}(x,n):=\sum_{v=0}^n(-1)^v(T^vl)(x),$$
where $l(x):=\log\left(\frac{1}{x}\right)$,
$$D(x,n):=\mathcal{L}(x,n)-l(x).$$
\end{definition}
We recall the following definitions from \cite{mrasympt}.
\begin{definition}\label{def28}
For $\lambda\geq 0$, we set
\begin{align*}
&A(\lambda):=\int_0^\infty\{t\}\{\lambda t\}\frac{dt}{t^2}\:,\\
&F(x):=\frac{x+1}{2}A(1)-A(x)-\frac{x}{2}\log x\:,\\
&H(x):=2\sum_{j\geq 0}(-1)^j\beta_{j-1}(x)F(\alpha_j(x))\:,\\
&B_1(t):=t-\lfloor t\rfloor -1/2,\ \text{the first Bernoulli function}\:,\\
&B_2(t):=\{t\}^2-\{t\}+1/6,\ (t\in\mathbb{R}) \ \text{the second Bernoulli function}\:.
\end{align*}
For $\lambda\in\mathbb{R}$, let 
$$\Phi_2(\lambda):=\sum_{n\geq 1}\frac{B_2(n\lambda)}{n^2}\:.$$
\end{definition}
\begin{lemma}\label{lem29}
It holds
$$A(\lambda)=\frac{\lambda}{2}\log\frac{1}{\lambda}+\frac{1+A(1)}{2}\:\lambda+
O(\lambda^2),\ \ \text{as}\ \lambda\rightarrow 0\:.$$
\end{lemma}
\begin{proof}
By \cite{balaz2}, Proposition 31, formula (74), we have:
$$A(\lambda)=\frac{\lambda}{2}\log\frac{1}{\lambda}+\frac{1+A(1)}{2}\:\lambda+\frac{\lambda^2}{2}\Phi_2\left(\frac{1}{\lambda}\right)-\int_{1/\lambda}^\infty\Phi_2(t)\frac{dt}{t^3}\:.$$
From Definition \ref{def28}, it follows that $\Phi_2(t)$ is bounded. 
Therefore
$$\frac{\lambda^2}{2}\Phi_2\left(\frac{1}{\lambda}\right)=O(\lambda^2)$$
and
$$\int_{1/\lambda}^\infty\Phi_2(t)\frac{dt}{t^3}=O(\lambda^2).$$
\end{proof}
\begin{lemma}\label{lem210}
We have
$$g(x)=l(x)+D(x,n)+H(x)+(-1)^{n+1}T^{n+1}\mathcal{W}(x).$$
\end{lemma}
\begin{proof}
From formula (3) of \cite{mr2} we have:
\[
\mathcal{W}(x)=\mathcal{L}(x,n)+(-1)^{n+1}T^{n+1}\mathcal{W}(x).\tag{2.1}
\]
In \cite{balaz2} the function $\Phi_1$ is defined by
\[
\Phi_1(t):=\sum_{n\geq 1}\frac{B_1(nt)}{n}=\sum_{n\geq 1}\frac{\{nt\}-1/2}{n}\:.\tag{2.1}
\]
Thus we have
\[
g(x)=-2\Phi_1(x)\:.\tag{2.2}
\]
By Proposition (2) of \cite{balaz2} we obtain
\[
\Phi_1(x)=-\frac{1}{2}\mathcal{W}(x)-\frac{1}{2}H(x)\tag{2.3}
\]
almost everywhere.\\
The proof of Lemma \ref{lem210} follows now from (2.1), (2.2), (2.3) 
and Definition \ref{def27}
\end{proof}
\section{Proof of Theorem \ref{thm11}}
\begin{definition}\label{def31}
Let $d, h\in \mathbb{N}_0$, $h\geq 1$, $u, v\in(0,\infty)$. Then we 
define
$$\mathcal{J}(d, h, u, v):=\{x\in X\::\: T^dl(x)\geq u\ \text{and}\ T^{d+h}l(x)\geq v\}\:.$$
\end{definition}
\begin{lemma}\label{lem32}
We have
$$m(\mathcal{J}(d, h, u, v))\leq 2\exp\left(-2^{\frac{h-2}{2}}v\exp\left(2^{\frac{d-2}{2}}u\right)\right)$$
\end{lemma}
\begin{proof}
This is Lemma 2.13 of \cite{mr2}.
\end{proof}
We recall the following definition from \cite{mr2}.
\begin{definition}\label{def34}(Definition 2.14 of \cite{mr2})\\
Let $L\in\mathbb{N}$. We set $j_0:=L-\left\lfloor\frac{L}{100}\right\rfloor$, $C_1:=1/400.$ For 
$j\in \mathbb{Z}$, $j\leq j_0$, we define the intervals:
$$I(L,j):=\left(x^{(j-1)}, x^{(j)}\right),\ \text{where } x^{(j)}:=\exp(-L+j).$$
For $v\in\mathbb{N}_0$, we set
$$a(L,v):=\exp(-C_1L+v)$$
$$\mathcal{T}(L, j, 0):=\{x\in I(L, j)\cap X\::\: |D(x,n)|\leq \exp(-C_1L)\}\:,$$
and for $v\in\mathbb{N}$, we set
$$\mathcal{T}(L, j, v):=\{x\in I(L, j)\cap X\::\: a(L, v-1)\leq |D(x,n)|\leq a(L, v)\}\:.$$
For $v, h\in\mathbb{Z}$, $h\geq 0$, we set
$$U(L, j, v, h):=\{x\in \mathcal{T}(L, j, v)\::\: T^hl(x)\geq 2^{-h}a(L, v-1)\}\:.$$
\end{definition}
\begin{lemma}\label{lem35}
There are constants $C_2, C_3>0$, such that for $v\geq 1$, we have
$$m(\mathcal{T}(L, j, v))\leq C_2\exp\left(-C_3\exp\left(-C_1L+v-1+\frac{1}{2}(L-j)\right)\right)\:.$$
\end{lemma}
\begin{proof}
This is lemma 2.15 of \cite{mr2}.
\end{proof}
\begin{definition}\label{def35} 
Let $L_0:=\lfloor K\rfloor+1$. For $j\in\mathbb{Z}$, $j\leq j_0$ we define:
$$\mathcal{E}_1(K,j,n):=\left\{x\in I(L_0,j)\::\: |D(x,n)|\geq \exp\left(-\frac{C_1}{2}\: K\right)\right\}$$
$$\mathcal{E}_2(K,j,n):=\left\{x\in I(L_0,j)\::\: |T^{n+1}\mathcal{W}(x)|\geq \exp\left(-\frac{C_1}{2}\: K\right)\right\}.$$
\end{definition}
\begin{lemma}\label{lem36}
For sufficiently large $K$ we have:
$$m(\mathcal{E}_1(K,j,n))\leq |I(L_0,j)|\exp(-K)\:.$$
\end{lemma}
\begin{proof}
We have
$$|D(x,n)|\geq \exp\left(-\frac{C_1}{2}K\right)$$
Therefore
$$x\in\bigcup_{v\geq \frac{C_1}{3}K}\mathcal{T}(L_0,j,v)$$
and thus by Lemma \ref{lem35} we have
$$m(\mathcal{E}_1(K,j,n))\leq C_2\sum_{v\geq \frac{C_1}{3}K}\exp\left(-C_3\exp\left(-C_1L_0+v-1+\frac{1}{2}(L-j)\right)\right)\leq |I_0(L_0,j)|\exp(-K).$$
\end{proof}

\begin{definition}\label{def37}
For $w\in\mathbb{N}_0$ we set 
$$\mathcal{V}(K,j,w,n):=\{x\in I(L_0,j)\::\: l(x)\exp\left(-\frac{C_1}{2}K+w\right)\leq |T^{(n+1)}\mathcal{W}(x)|\leq l(x)\exp\left(-\frac{C_1}{2}K+w+1\right)$$
$$\mathcal{Z}(K,j,v,w,n):=\mathcal{T}(L_0,j,v)\cap \mathcal{V}(K,j,w,n)\:.$$
\end{definition}
\begin{lemma}\label{lem38}
Here and in the sequel we assume that $n\geq n_0(K)$, where $n_0(K)$
is chosen sufficiently large. It holds
$$m(\mathcal{V}(K,j,w,n))\leq \exp(-2w)(L_0-j-H+w)^{-1}\exp(-\exp(4K))\:,$$
where 
$$H:=\sup_{x\in(0,1)} |H(x)|.$$
\end{lemma}
\begin{proof}
By Lemma \ref{lem26} we have
\begin{align*}
m(\mathcal{V}(K,j,w,n))(L_0-j-H+w)^2\exp(-C_1K+2w)&\leq \int_{\mathcal{V}(K,j,w,n)}|T^{n+1}\mathcal{W}(x)|^2dm(x)\\
&\leq g^{2(n-1)}\int_0^1|\mathcal{W}(x)|^2dm(x).
\end{align*}
Thus 
$$m(\mathcal{V}(K,j,w,n))\leq g^{2(n-1)}\left(\int_0^1\mathcal{W}(x)^2dm(x)\right)\exp\left(\frac{C_1}{2}K-2w\right)(L_0-j-H+w)^{-2}.$$
The result of Lemma \ref{lem38} follows by choosing $n$ sufficiently large.
\end{proof}
\begin{lemma}\label{lem310}
We have for $n\geq n_0(K)$:
$$\int_{\mathcal{E}_1(K,j,n)}|g(x)|^Kdx\leq |I(K,j)|\exp(-K).$$
\end{lemma}
\begin{proof}
We have 
$$\mathcal{E}_1(K,j,n)\subseteq \bigcup_{v\geq1}\mathcal{T}(K,j,v)$$
and therefore by Definition \ref{def34}:
\begin{align*}
\int_{\mathcal{E}_1(K,j,n)}|g(x)|^Kdx&\leq \sum_{v\geq 1} \int_{\mathcal{T}(K,j,v)}|g(x)|^Kdx\leq \sum_{v\geq 1}m(\mathcal{T}(K,j,v))a(K,v)^K\\
&\leq |I(K,j)|\exp(-K)
\end{align*}
by Lemma \ref{lem35}.
\end{proof}
\begin{lemma}\label{lem39}
$$m(\mathcal{E}_2(K,j,n))\leq \exp(-\exp(3K))\:.$$
\end{lemma}
\begin{proof}
This follows from Definition \ref{def35} and Lemma \ref{lem38}.
\end{proof}
\begin{lemma}\label{lem311}
We have for $n\geq n_0(K)$:
$$\int_{\mathcal{E}_2(K,j,n)}|g(x)|^Kdx\leq (|j|+1)^{-2}\exp(-K)$$
\end{lemma}
\begin{proof}
For $x\in \mathcal{Z}(K,j,v,w,n)$ we have:
$$|g(x)|\leq b(x,K,j,n)+w+1+|D(x,n)|,$$
where $b(x,K,j,n):=l(x)+L_0-j+w+1$. Thus
\begin{align*}
\int_{\mathcal{Z}(K,j,v,w,n)}|g(x)|^Kdx\leq &2^K\left(\sup_{x\in I(K,j)}|b(x,K,j,n)|^K
+|I(L_0,j)|l(x^{(j-1)})^K\exp\left(-\frac{C_1}{2}K^2+(w+1)K\right)\right)\\
&\times(m(\mathcal{T}(K,j,v)+m(\mathcal{V}(K,j,w,v,n)).
\end{align*}
The result follows by summation over $v$ and $w$.
\end{proof}
\begin{definition}\label{def312}
We set
$$x_0:=\exp\left(-\left\lfloor \frac{L_0}{100}\right\rfloor\right)\:.$$
\end{definition}
\begin{lemma}\label{lem313}
There is a contant $C_4>0$, such that
$$\int_{x_0}^{1/2}|g(x)|^Kdx\leq \Gamma(K+1)\exp(-C_4K)\:.$$
\end{lemma}
\begin{proof}
We apply Lemma 2.22 of \cite{mr2} with $L=L_0:=\lfloor K\rfloor +1$. There
is a constant $C_5^*>0$, such that
\[
\int_{x_0}^{1/2}|\mathcal{L}(x,n)|^{L_0}dx\leq \Gamma(L_0+1)\exp(-C_5^*L_0).
\]
By Definitions \ref{def27}, \ref{def28} and Lemma \ref{lem210} we get
\[
g(x)=\mathcal{L}(x,n)+H(x)+(-1)^{n+1}T^{n+1}\mathcal{W}(x)\:.\tag{3.2}
\]
Thus by (3.2) using the notation 
$$\|h\|_{L_0}:=\left(\int_{x_0}^{1/2}|h(x)|^{L_0}dx\right)^{1/L_0}$$
we obtain
\begin{align*}
\left(\int_{x_0}^{1/2}|g(x)|^{L_0}dx\right)^{1/L_0}&\leq (\Gamma(L_0+1)\exp(-C_5^*L_0))^{1/L_0}+\|H(x)\|_{L_0}+\|T^{n+1}\mathcal{W}(x)\|_{L_0}\\
&=(\Gamma(L_0+1)\exp(-C_5^*L_0))^{1/L_0}\left(1+O\left(\frac{1}{\Gamma(L_0+1)^{1/L_0}}\right)\right).
\end{align*}
We obtain
$$\int_{x_0}^{1/2}|g(x,n)|^{L_0}dx=O(\Gamma(L_0+1)\exp(-C_5^*L_0)).$$
The result of Lemma \ref{lem313} follows, since
$$\Gamma(L_0+1)=O(K\Gamma(K+1)).$$
\end{proof}
\begin{definition}\label{def314}
We set
$$\mathcal{A}:=((0,x_0)\cap X)-\bigcup_{j\leq j_0}\mathcal{E}_1(K,j,n)-\bigcup_{j\leq j_0}\mathcal{E}_2(K,j,n).$$
\end{definition}
\begin{lemma}\label{lem315}
There is a constant $C_5>0$, such that
$$\int_0^{1/2}|g(x)|^Kdx=\int_\mathcal{A} g(x)^Kdx+O(\Gamma(K+1)\exp(-C_5K)).$$
\end{lemma}
\begin{proof}
By Definition \ref{def314} we have
\begin{align*}
\int_0^{1/2}|g(x)|^Kdx=&\int_{\mathcal{A}}|g(x)|^Kdx+\sum_{j\leq j_0}\int_{\mathcal{E}_1(K,j,n)}|g(x)|^Kdx\\
&+\sum_{j\leq j_0}\int_{\mathcal{E}_2(K,j,n)}|g(x)|^Kdx+\int_{x_0}^{1/2}|g(x)|^Kdx.\tag{3.3}
\end{align*}
From Lemmas \ref{lem310}, \ref{lem311} and \ref{lem313} we obtain
\[
\int_{(0,1/2)\setminus \mathcal{A}}|g(x)|^K dx=O(\Gamma(K+1)\exp(-C_4K)).\tag{3.4}
\]
Therefore from (3.4) we get
\[
\int_0^{1/2}|g(x)|^Kdx=\int_{\mathcal{A}}|g(x)|^Kdx+O(\Gamma(K+1)\exp(-C_4K)).\tag{3.5}
\]
By Definition \ref{def35} for $\mathcal{E}_1(K,j,n)$, $\mathcal{E}_2(K,j,n)$
and Lemma \ref{lem210} we have $|g(x)|=g(x)$ for $x\in \mathcal{A}$.
Lemma \ref{lem315} follows from (3.5).
\end{proof}
\begin{definition}\label{def316}
For $x\in X$ let
$$R(x,n):=(D(x,n)+H(x)+(-1)^{n+1}T^{n+1}\mathcal{W}(x))l(x)^{-1}.$$
\end{definition}
\begin{lemma}\label{lem317}
There is a constant $C_6>0$, such that for $x\in\mathcal{A}$ we have
$$R(x,n)=(\gamma-2\pi)l(x)^{-1}+O(\exp(-C_6K)).$$
\end{lemma}
\begin{proof}
By the Definition \ref{def35} for $\mathcal{E}_1(K,j,n)$, $\mathcal{E}_2(K,j,n)$ and Definition \ref{def314} for $\mathcal{A}$, we have for $x\in\mathcal{A}$
\[
|D(x,n)|<\exp\left(-\frac{C_1}{2}K\right).\tag{3.6}
\]
\[
|T^{n+1}\mathcal{W}(x)|<\exp\left(-\frac{C_1}{2}K\right).\tag{3.7}
\]
By Definition \ref{def28} we have
$$H(x)=2\sum_{j\geq 0}(-1)^j\beta_{j-1}(x)F(\alpha_j(x)),$$
where
$$F(x):=\frac{x+1}{2}A(1)-A(x)-\frac{x}{2}\log x.$$
By Lemma \ref{lem29} we have
$$A(x)=\frac{x}{2}\log\frac{1}{x}+\frac{1+A(1)}{2}\: x+O(x^2).$$
From $\beta_{j-1}=\alpha_1(x)\cdots \alpha_{j-1}(x)$ with $\alpha_0(x)=x$,
$|\alpha_l(x)|\leq 1$ and $x\leq x_0$ it follows that
\[
H(x)=-A(1)+O\left(\exp\left(-\frac{K}{200}\right)\right).\tag{3.8}
\]
\end{proof}
In \cite{baez} it is proved at page 225 that
\[
A(1)=\log2\pi-\gamma.\tag{3.9}
\]
Lemma \ref{lem317} now follows from (3.6), (3.7), (3.8) and (3.9).
\begin{lemma}\label{lem318}
For $x\in\mathcal{A}$ we have
$$|g(x)|^K=g(x)^K=l(x)^K\sum_{j=0}^\infty\binom{K}{j}R(x,n)^j\:.$$
\end{lemma}
\begin{proof}
By Lemma \ref{lem210} we have for $x\in\mathcal{A}$:
$$g(x)=l(x)(1+R(x,n)),\ \text{where } |R(x,n)|<1,$$
by Lemma \ref{lem317}. The result of Lemma \ref{lem318} thus follows from the Binomial Theorem for real exponents.
\end{proof}
\begin{lemma}\label{lem319}
For $x\in\mathcal{A}$ we have
$$g(x)^K=l(x)^K\left(\sum_{j=0}^{\lfloor K\rfloor}\binom{K}{j}R(x,n)^j+O(\exp(-K))\right)\:.$$
\end{lemma}
\begin{proof}
From Definition \ref{def312} we have for $x\in(0,x_0)$:
$$|l(x)|\geq cK\ \text{with an absolute constant c>0.}$$
By Definitions \ref{def35} and \ref{def314} we have for $x\in \mathcal{A}$:
$$|R(x,n)|\leq BK^{-1},$$
where $B>0$ is an absolute constant.
For $0\leq j\leq \lfloor K\rfloor$ we have
$$\binom{K}{j}<\binom{\lfloor K\rfloor+1}{j}\leq 2^{K+1}.$$
For $j=\lfloor K\rfloor+h$, $h\in\mathbb{N}$ we have
\[
\binom{K}{j}\leq \binom{K}{\lfloor K\rfloor}\frac{2}{\lfloor K\rfloor+1}\cdots\frac{h+1}{\lfloor K\rfloor+h}\leq 2^{K+1}\tag{3.7}
\]
From (3.6) and (3.7) we obtain
\[
\sum_{j>\lfloor K\rfloor}\binom{K}{j}R(x,n)^j\leq 2^{K+1}\frac{(BK^{-1})^{K-1}}{1-(BK^{-1})}\:.\tag{3.8}
\]
Lemma \ref{lem319} follows from Lemma \ref{lem318} and (3.8).
\end{proof}
\begin{lemma}\label{lem320}
There is an absolute constant $C_7>0$, such that for $K/2<L\leq K$, $L\in\mathbb{R}$, we have:
\[
\int_{\mathcal{A}}l(x)^Ldx=\Gamma(L+1)(1+O(\exp(-C_7K)))\:.\tag{3.9}
\]
\end{lemma}
\begin{proof}
By Definition \ref{def314} we have:
\begin{align*}
\int_{(0,1/2)\setminus\mathcal{A}}l(x)^Ldx=&\sum_{j\leq j_0}\left(\int_{\mathcal{E}_1(K,j,n)}l(x)^Ldx+\int_{\mathcal{E}_2(K,j,n)}l(x)^Ldx\right)\tag{3.10}\\
&+\int_{x_0}^{1/2}l(x)^Ldx\:.
\end{align*}
There are absolute constants $c_1>0$, $c_{2,i}>0$ (i=1,2) such that
$$\min_{x\in I(K,j)} l(x)^L\leq \max_{x\in I(K,j)} l(x)^L\leq c_1\min_{x\in I(K,j)} l(x)^L$$
$$\min_{x\in \mathcal{E}_i(K,j,n)} l(x)^L\leq \max_{x\in \mathcal{E}_i(K,j,n)} l(x)^L\leq c_2\min_{x\in \mathcal{E}_i(K,j,n)} l(x)^L.$$
Therefore there is an absolute constant $c_3>0$, such that
\[
\int_{\mathcal{E}_i(K,j,n)}l(x)^Ldx\leq c_3\left(\int_{I(K,j)}l(x)^Ldx\right)\: m(\mathcal{E}_i(K,j,n)).\tag{3.11}
\]
By summation over $j$ and Lemma \ref{lem36}.
\[
\sum_{j\leq j_0}\int_{\mathcal{E}_1(K,j,n)}l(x)^Ldx\leq \Gamma(L+1)\exp(-K).\tag{3.12}
\]
We have
$$\sum_{j\leq j_0}\int_{\mathcal{E}_2(K,j,n)}l(x)^Ldx=\Sigma^{(1)}+\Sigma^{(2)}\ \ ,$$
where
$$\Sigma^{(1)}:=\sum_{-\exp(e^K)\leq j\leq j_0}\int_{\mathcal{E}_2(K,j,n)}l(x)^Ldx,$$
$$\Sigma^{(2)}:=\sum_{j<-\exp(\exp(K))}\int_{\mathcal{E}_2(K,j,n)}l(x)^Ldx.$$
For $$-\exp(e^K)\leq j\leq j_0$$ we have by Lemma \ref{lem39}:
$$m(\mathcal{E}_2(K,j,n))\leq |I(K,j)|\exp(-\exp(3K)).$$
Therefore
\[
\Sigma^{(1)}\leq \left(\int_{\exp(-L_0-\exp(e^K))}^\infty l(x)^Ldx\right)\exp(-\exp(3K))\leq \exp(-\exp(3K))\Gamma(L+1).\tag{3.13}
\]
We have
\[
\Sigma^{(2)}\leq \int_0^{\exp(L_0-\exp(e^K))}l(x)^Kdx\leq \Gamma(L+1)\exp\left(-\exp\left(\frac{K}{2}\right)\right)\tag{3.14}
\]
We also have
\[
\int_{x_0}^{1/2}l(x)^Ldx=\Gamma(L+1)\exp(-C_7'K)\tag{3.15}
\]
for an appropriate constant $C_7'>0$, if $K/2<L\leq K$.\\
Lemma \ref{lem320} now follows from (3.9)--(3.15).
\end{proof}
\begin{lemma}\label{lem321}
There is an absolute constant $C_8>0$, such that for $0\leq j\leq \lfloor K\rfloor$ we have:
$$\int_{\mathcal{A}}l(x)^KR(x,n)^jdx=D^j\left(\int_{\mathcal{A}}l(x)^{K-j}dx\right)(1+O(\exp(-C_8K)),$$
where $D:=\gamma-2\pi$.
\end{lemma}
\begin{proof}
By Lemma \ref{lem317} we have
\[
R(x,n)=Dl(x)^{-1}+Q(x),\tag{3.16}
\]
where $$Q(x):=O(\exp(-C_6K))\:.$$
By the Binomial Theorem we have
\[
\int_{\mathcal{A}}l(x)^KR(x,n)^jdx=D^j\int_{\mathcal{A}}l(x)^{K-j}dx+\sum_{h=1}^j\binom{j}{h}D^{j-h}\int_{\mathcal{A}}l(x)^{K-j+h}Q(x)^hdx.\tag{3.17}
\]
We have
$$\int_{\mathcal{A}}l(x)^{K-j+h}Q(x)^hdx\leq \Gamma(K-j+h+1)\exp(-C_6hK),$$
where 
$$\Gamma(K-j+h)\leq K^h\Gamma(K-j+1)\:.$$
Therefore
$$\int_{\mathcal{A}}l(x)^{K-j+h}Q(x)^hdx\leq (K\exp(-C_6K)^h)\Gamma(K-j+1)$$
and Lemma \ref{lem321} follows from (3.17).
\end{proof}
\section{Proof of Theorem \ref{thm11}}
By Lemma \ref{lem315} we have
\[
\int_0^{1/2}|g(x)|^Kdx=\int_{\mathcal{A}}g(x)^Kdx+O(\Gamma(K+1)\exp(-C_5K)).\tag{3.18}
\]
From Lemma \ref{lem319} we obtain
\[
\int_{\mathcal{A}}g(x)^Kdx=\sum_{j=0}^{\lfloor K\rfloor}\binom{K}{j}\int_{\mathcal{A}}l(x)^KR(x,n)^jdx+O(\Gamma(K+1)\exp(-K)).\tag{3.19}
\]
From Lemmas \ref{lem320} and \ref{lem321}, formulas (3.18) and (3.19), we obtain:
\begin{align*}
\tag{3.20}\int_0^{1/2}|g(x)|^Kdx=&\sum_{j\leq K/2}D^j\binom{K}{j}\Gamma(K-j+1)(1+O(\exp(-C_7K))\\
&+\sum_{K/2<j\leq K}\binom{K}{j}\int_{\mathcal{A}}l(x)^KR(x,n)^jdx+O(\Gamma(K+1)\exp(-K))\:.
\end{align*}
We have
\[
\binom{K}{j}\Gamma(K-j+1)=\frac{1}{j!}\Gamma(K+1)\tag{3.21}
\]
and
\[
\sum_{K/2<j\leq K}\binom{K}{j}=O\left(\Gamma(K+1)\sum_{j>K/2}\frac{1}{j!}\right)=O(\Gamma(K+1)\exp(-K/2)).\tag{3.22}
\]
Theorem \ref{thm11} now follows from (3.19), (3.20), (3.21), and (3.22).\qed

%
%
%
%
%
%

\end{document}